\documentclass[11pt]{article}
\date{}

\usepackage[title]{appendix}
\usepackage{xcolor}
\usepackage[margin=1in]{geometry}                
\usepackage{graphicx}
\usepackage{subcaption}
\usepackage{pdflscape}
\usepackage{amssymb}
\usepackage[normalem]{ulem}
\usepackage{hyperref}
\usepackage{enumitem}
\usepackage{mathrsfs}
\usepackage{epstopdf}
\usepackage{rotating}
\usepackage{longtable} 
\usepackage{adjustbox}
\usepackage{float}

\usepackage{color}
\usepackage{bbm, dsfont}
\usepackage{pst-node}
\usepackage{tikz-cd}
\usepackage{amsfonts} 
\usepackage{geometry}
\usepackage{amsthm}
\usepackage{amsmath}
\usepackage{titlesec}
\usepackage{amssymb}
\usepackage{enumitem}
\usepackage{float}
\usepackage [english]{babel}
\usepackage [autostyle, english = american]{csquotes}
\usepackage{algorithm}
\usepackage[noend]{algpseudocode} 
\makeatletter
\def\BState{\State\hskip-\ALG@thistlm}
\makeatother
\usepackage{hyperref}
\usepackage[normalem]{ulem}
\usepackage{mathrsfs}
\usepackage[italicdiff]{physics}

\newlist{casess}{enumerate}{1}
\setlist[casess]{label=     \textbf{Case} \arabic*:}
\usepackage{mathtools}

\makeatletter
\newcommand*{\rom}[1]{\expandafter\@slowromancap\romannumeral #1@}
\makeatother

\usepackage{etoolbox}

\makeatletter
\patchcmd{\ttlh@hang}{\parindent\z@}{\parindent\z@\leavevmode}{}{}
\patchcmd{\ttlh@hang}{\noindent}{}{}{}
\makeatother

\usepackage{listings}
\usepackage{color} 
\definecolor{mygreen}{RGB}{28,172,0} 
\definecolor{mylilas}{RGB}{170,55,241}

\newlist{Assumptions}{enumerate}{1}
\setlist[Assumptions]{label=     \textbf{Assumption} \arabic*:}

\makeatletter

\newsavebox{\@brx}
\newcommand{\llangle}[1][]{\savebox{\@brx}{\(\m@th{#1\langle}\)}%
  \mathopen{\copy\@brx\kern-0.5\wd\@brx\usebox{\@brx}}}
\newcommand{\rrangle}[1][]{\savebox{\@brx}{\(\m@th{#1\rangle}\)}%
  \mathclose{\copy\@brx\kern-0.5\wd\@brx\usebox{\@brx}}}
\makeatother

\usepackage{lipsum} 
\usepackage{titlesec}
\titleformat{\subsection}[runin]
       {\normalfont\bfseries}
       {\thesubsection}
       {0.5em}
       {}
       [.]

 \newtheorem{thm}{Theorem}[section]

 \newtheorem{lem}[thm]{Lemma}
 \newtheorem{prop}[thm]{Proposition}
 \theoremstyle{definition}
 \newtheorem{defn}[thm]{Definition}
 \theoremstyle{remark}
 \newtheorem{rem}[thm]{Remark}

 \numberwithin{equation}{section}

\numberwithin{equation}{section}

\def\N{\mathbb{N}}

\DeclarePairedDelimiterX{\inp}[2]{\langle}{\rangle}{#1, #2}

\makeatletter
\newcommand*\bigcdot{\mathpalette\bigcdot@{.5}}
\newcommand*\bigcdot@[2]{\mathbin{\vcenter{\hbox{\scalebox{#2}{$\m@th#1\bullet$}}}}}
\makeatother

\def\<{\langle}
\def\>{\rangle}

\newcommand{\Cs}{\ensuremath{\mathrm{C}^\ast}}

\numberwithin{equation}{section}

\usepackage[backend=biber,maxnames=10,style=alphabetic]{biblatex}
\begin{document}

\title{{\rm PHP} decompositions and primitivity of group rings}

\author{Felipe I. Flores
\footnote{
\textbf{2020 Mathematics Subject Classification:} Primary 16S34, Secondary 20C07, 20F65, 20G15.
\newline
\textbf{Key Words:} Primitive, group ring, property $\rm PHP$, linear group, trivial amenable radical.}
}

\maketitle

\begin{abstract}\setlength{\parindent}{0pt}\setlength{\parskip}{1ex}\noindent
We show that every countable group $G$ that has Ozawa's property $\rm PHP$ and contains a non-abelian free subgroup has the following property: the group ring $KG$ is primitive for any field $K$. As a consequence, we greatly extend the previously existing results on the primitivity of group rings. New examples include linear groups and groups of homeomorphisms of the Cantor set. Along the way, we prove that the class of groups with $\rm PHP$ enjoys some permanence properties that are of independent interest.
\end{abstract}

\section{Introduction}

A ring $R$ is said to be (right) primitive if it has a faithful irreducible (right) $R$-module. For many years, it was not known whether there exists a group $G$ and a ring $R$ such that the group ring $RG$ is primitive \cite[Problem 17]{Ka70}. In \cite{FoSn72}, Formanek and Snider provided the first example of such a group, and they even showed that these groups exist in abundance: any group can be embedded in another group $G$ such that $KG$ is primitive for some field $K$.

Since then, the study of primitivity in group rings has been an active area of research for several decades. A major milestone was achieved in 1978 through the work of Domanov \cite{Do78}, Farkas–Passman \cite{FaPa78}, and Roseblade \cite{Ro78,Ro79}, who established a complete characterization of the primitivity of group rings of polycyclic-by-finite groups. More concretely, they proved that, for a polycyclic-by-finite group $G$, the group ring $KG$ is primitive if and only if the FC-center of $G$ is trivial and the field $K$ is not absolute.

Subsequent work addressed non-noetherian groups, producing results on free products \cite{Fo73}, amalgamated free products and HNN extensions \cite{Bo89,Ni07,AlNi17}, locally free groups \cite{Ni11}, torsion-free (non-elementary) hyperbolic groups \cite{So18}, and acylindrically hyperbolic groups without finite normal subgroups \cite{AbDa19}. 

Very recently, inspired by the study of \Cs-algebra theory, the author of the present article showed that the group ring $KG$ is primitive whenever $G$ admits a topologically free, extremely proximal action $G\curvearrowright X$. This result provided new and interesting examples of primitive group rings, such as those associated with commensurator groups of non-elementary torsion-free hyperbolic groups and lim-free weakly hyperbolic groups.

As with \cite{Fl26}, the approach taken in the present work is based on Alexander-Nishinaka's work \cite{AlNi17} and is still motivated by the author's work on \Cs-algebra theory. One of the main differences is how deeply we delve into the techniques developed in \cite{AlNi17}. Indeed, most of the recent work (see \cite{So18,AbDa19,Fl26}) fundamentally relies on verifying a combinatorial/group-theoretic criterion explicitly stated in \cite[Theorem 1.1]{AlNi17}. This criterion could be described as a strong form of mixed-identity-freeness. Our approach seeks to establish primitivity using properties that are more intrinsic to the ring structure of $KG$, as opposed to only the group-theoretic properties of $G$. It is also fundamentally based on the groundbreaking insight developed by Ozawa \cite{Oz25}, who attained selflessness by cleverly constructing a free semicircular element from the complete algebraic structure, rather than simply using group elements. This insight led to the definition of property ${\rm PHP}$ and the suppression of the rapid decay assumption that was present in previous works surrounding selflessness. In that regard, we view our approach as a purely algebraic parallel to the recent developments in the theory of $C^*$-algebras. By now, many natural and interesting classes of groups are known to enjoy property $\rm PHP$; for example, we refer to \cite{Oz25,Vi26,FKOCP26,BaFl26,Ry26, Ar26}.

We will postpone stating the precise definition of property $\rm PHP$, as it might be somewhat technical and not particularly enlightening at first sight (see Definition \ref{thedef}). For the moment, let us intuitively mention that it is a combinatorial property inspired by north-south dynamics and the classical paradoxical decomposition introduced by Powers \cite{Po75}. However, this property is flexible enough to be enjoyed by groups that do not admit actions with north-south dynamics in any traditional sense, such as $\mathbb F_2\times \mathbb F_2$. For more examples of this phenomenon, we refer the reader to our result on linear groups (see Theorem \ref{maincor}).

The main result of this article is as follows:

\begin{thm}\label{mainthm}
   Let $G$ be a group with Ozawa's property ${\rm PHP}$ that contains a non-abelian free subgroup whose cardinality is the same as that of $G$. Then, if $R$ is a domain with $|R|\leq |G|$, the group ring $RG$ of $G$ over $R$ is primitive. Moreover, the group ring $KG$ is primitive for any field $K$.
\end{thm}

The above theorem applies to many classes of groups, which allows us to describe many applications of our main theorem. The first one is to linear groups. Recall that a group $G$ is said to be \emph{linear} if it embeds into a group of matrices over some field $k$. 

For linear groups, we use a result of Vigdorovich \cite{Vi26}, who characterized the class of finitely generated linear groups with property ${\rm PHP}$ in terms of the amenable radical. Recall that a group is amenable if its left-translation action admits an invariant mean. In the case of linear groups, and thanks to the Tits alternative \cite{Ti72}, nonamenability is particularly easy to state: for finitely generated linear groups, it is equivalent to having a non-abelian free subgroup. Hence, a linear group has a trivial amenable radical if and only if every normal subgroup contains a non-abelian free subgroup. 

\begin{thm}\label{maincor}
Let $G$ be a non-trivial countable linear group with trivial amenable radical. Then, if $R$ is a countable domain, the group ring $RG$ is primitive. Moreover, the group ring $KG$ is primitive for any field $K$.
\end{thm}

Another interesting class of examples comes from groups acting on the Cantor set. Indeed, we use $2^\N$ to denote the standard Cantor set. A subgroup $G\subset {\rm Homeo}(2^\N)$ is called \emph{vigorous} if, for all non-empty clopen subsets $A,B,C\subset 2^\N$ with $B,C\subsetneq A$, there exists $g\in G$ such that $g|_{X\setminus A} = id|_{X\setminus A}$ and $gB \subset C$. The class of vigorous groups is vast and contains many interesting groups \cite{BlElHy24}. Vigorous groups were shown to have property $\rm PHP$ by Arimoto \cite{Ar26}. Some examples include the Higman–Thompson groups $G_{n,r}$, the Brin-Thompson groups $nV$, and R\"over’s group $V(\Gamma)$.

\begin{thm}\label{maincor2}
Let $G$ be a countable vigorous group. Then, if $R$ is a countable domain, the group ring $RG$ is primitive. Moreover, the group ring $KG$ is primitive for any field $K$.
\end{thm}

Finally, as our last application, we recover the main theorem of \cite{Fl26}. These groups were shown to have property $\rm PHP$ in Ozawa's original paper \cite{Oz25}. They are characterized by having an action with properties that we define now. Let $G\curvearrowright X$ be an action by homeomorphisms on the Hausdorff space $X$. Assume that $\# X>2$. This action is called {extremely proximal} if, for every pair of non-empty open subsets $U,V\subseteq X$, there exists $g\in G$ such that $g(X\setminus U)\subseteq V$. The same action is called \emph{topologically free} if, for every $g\in G\setminus\{1\}$, the set of points fixed by $g$ is nowhere dense.

Our next result recovers \cite[Theorem 1.2]{Fl26}. In particular, it also applies to all the relevant classes of groups studied in \cite{AbDa19,BrIvOm20,FLMS,Yan25,BaFl26,Ry26}.

\begin{thm}[{\cite{Fl26}}]\label{maincor3}
   Let $G$ be a countable group that admits a topologically free, extremely proximal action $G\curvearrowright X$. Then, if $R$ is a countable domain, the group ring $RG$ is primitive. Moreover, the group ring $KG$ is primitive for any field $K$.
\end{thm}

Let us now comment on the proofs. As we mentioned before, this article follows the methods of Alexander-Nishinaka \cite{AlNi17} closely. To be more precise, we also use support-expansion techniques combined with Formanek's criterion \cite[Theorem 1]{Fo73} to conclude primitivity, resembling the ideas behind the proof of \cite[Theorem 1.1]{AlNi17}. Furthermore, our Lemma \ref{AN} is a direct consequence of \cite[Lemma 3.5]{AlNi17}, so that part is almost the same. The main difference is that our main support-expansion lemma is based on property $\rm PHP$ instead of Alexander-Nishinaka's property $(*)$, and we use it in a way that resembles a MIF-like property for the ring $RG$ instead of for the group $G$, in the same way that selflessness replaces mixed-identity-freeness in the \Cs-algebraic context. This difference can be appreciated in Lemma \ref{exp} and its proof. Any reader that is familiar with Ozawa's results on selflessness \cite{Oz25} will recognize the construction of a free semicircular element from property $\rm PHP$.

Along the way, we also prove some permanence properties for the class of groups with property $\rm PHP$. These are of independent interest, and their relevance will be explained in the next section.

\section{Property {\rm PHP}}

The purpose of this section is to introduce Ozawa's property ${\rm {PHP}}$ \cite{Oz25}. We will also prove two permanence properties that are of independent interest.

\begin{defn}[Ozawa]\label{thedef}
We say that a group $G$ has the property ${\rm {PHP}}$ if, for every finite $F\subseteq G$ and every $\varepsilon>0$, there exist $N\geq 1$ such that for all $n\geq N$, there are elements $t_1,\ldots,t_n\in G$ and subsets $C_i\subseteq D_i\subseteq G$ such that:
\begin{enumerate}
\item the members of
$$
\{aC_i:a\in F,\ 1\leq i\leq n\}
 \ \cup\
 \{at_i^{-1}( G\setminus D_i):a\in F,\ 1\leq i\leq n\}
$$
are pairwise disjoint;
\item
for every $x\in G$, the set
$$
 J(x):=\{i:x\in D_i\}
   \cup\{i:x\in t_i^{-1}( G\setminus C_i)\}
$$
has cardinality at most $\varepsilon\sqrt n$.
\end{enumerate}
\end{defn}

We first prove that the property ${\rm {PHP}}$ is stable under directed unions. In the future, we will use it to conclude that a group has property ${\rm {PHP}}$ whenever all of its finitely generated subgroups do. This was avoided in \cite[Proof of Theorem 1.4]{Vi26}.

\begin{lem}\label{directed}
Suppose $ G$ is the directed union of subgroups $ G_j$, each of which has property ${\rm {PHP}}$.  Then $ G$ has property ${\rm {PHP}}$.
\end{lem}

\begin{proof}
Fix a finite subset $F\subseteq G$ and a tolerance $\varepsilon>0$, and choose $j$ with $F\subseteq G_j=:H$. Let $T$ be a right transversal for $H$ in $ G$, such that every element of $ G$ has a unique decomposition $ht$, with $h\in H$ and $t\in T$.

For witnesses $t_i\in H$ and $C_i\subseteq D_i\subseteq H$ to ${\rm {PHP}}$ inside $H$, define
$$
 C_i'=C_iT,\qquad D_i'=D_iT.
$$
Then
$$
  G\setminus D_i'=(H\setminus D_i)T,
$$
because left translations in the PHP conditions preserve each right coset $Ht$. Thus, pairwise disjointness inside $H$ gives pairwise disjointness in every right coset; hence in $ G$.

If $x=ht$ is the unique right-coset expression, then
$$
 x\in D_i'\Longleftrightarrow h\in D_i, \quad x\in t_i^{-1}( G\setminus C_i')\Longleftrightarrow h\in t_i^{-1}(H\setminus C_i).
$$
The multiplicity bound is, therefore, unchanged. The same $N$ works for $F$ in $ G$, proving property ${\rm {PHP}}$.
\end{proof}

We now show that property ${\rm {PHP}}$ is stable under passage to finite index subgroups. This was proven in \cite{FKOCP26} for groups admitting topologically free extreme boundary actions, and it is relevant for \cite{Be26}, where Bell works with groups whose finite-index subgroups give selfless \Cs-algebras. Furthermore, in \cite[p. 18]{Vi26}, Vigdorovich asks if selflessness passes to the \Cs-algebras of finite-index subgroups. The next result, combined with \cite{Oz25}, provides a positive answer in the presence of property ${\rm {PHP}}$.

\begin{prop}
    Suppose that $G$ has property ${\rm {PHP}}$ and that $H\leq G$ has finite index. Then $H$ has property ${\rm {PHP}}$.
\end{prop}
\begin{proof}
    Let $d=[G:H]$ and choose a set $S$ of representatives for the left cosets of $H$, so that
$$
  G=\bigsqcup_{s\in S}sH,
  \qquad |S|=d.
$$
Fix a finite subset $F\subseteq H$ and $\varepsilon>0$, and put
$$
  \widetilde F=F\cup\bigcup_{s\in S}\bigl(sFs^{-1}\cup sF\bigr), \qquad \delta=\frac{\varepsilon}{2\sqrt d}.
$$
By property ${\rm {PHP}}$ for $G$, there is an $N\geq1$ such that the property is witnessed for $(\widetilde F,\delta)$ and for every integer $n\geq N$.

Let $m\geq\lceil N/d\rceil$ be arbitrary and set $n=dm$. For $1\leq i\leq n$, fix the witnesses $t_1,\ldots,t_n\in G$ and $C_i\subseteq D_i\subseteq G$ for $(\widetilde F,\delta)$. Since there are $d$ cosets, at least $m$ of the elements $t_i$ lie in the same coset $sH$ for some $s\in S$. Retain exactly $m$ such indices, relabel them as $1,\ldots,m$, and set
$$
  h_i=s^{-1}t_i\in H, \qquad A_i=s^{-1}C_i\cap H, \qquad B_i=s^{-1}D_i\cap H.
$$
We now verify that $h_i,A_i,B_i$ are ${\rm {PHP}}$ witnesses for $(F,\varepsilon)$ with $m$ indices. First, it is obvious that $A_i\subseteq B_i$, $sA_i=C_i\cap sH$, and $s(H\setminus B_i)=sH\setminus D_i$. Furthermore, for $a\in F$, we have
$$
  s(aA_i)=(sas^{-1})(C_i\cap sH)\subseteq(sas^{-1})C_i
$$
and 
$$
s\bigl(ah_i^{-1}(H\setminus B_i)\bigr)=(sa)t_i^{-1}(sH\setminus D_i)\subseteq (sa)t_i^{-1}(G\setminus D_i).
$$
Since $sFs^{-1},sF\subseteq \widetilde F$ and left multiplication by $s$ is a bijection, the sets $aA_i$ are pairwise disjoint over distinct pairs $(a,i)$, and the sets $b h_i^{-1}(H\setminus B_i)$ are also pairwise disjoint over distinct pairs $(b,i)$. Finally, for every $a,b\in F$ and $1\leq i,j\leq m$,
$$
  s\bigl(aA_i\cap b h_j^{-1}(H\setminus B_j)\bigr)\subseteq (sas^{-1})C_i\cap(sb)t_j^{-1}(G\setminus D_j)=\emptyset,
$$
again, because $sas^{-1},sb\in\widetilde F$. This proves all the required disjointness relations.

For the multiplicity estimate, let $x\in H$. By the definition of $B_i$ and noting that $h_i^{-1}(H\setminus A_i)=t_i^{-1}(sH\setminus C_i)$, we have
\begin{align*}
  &\left|\{1\leq i\leq m:x\in B_i\}\cup\{1\leq i\leq m:x\in h_i^{-1}(H\setminus A_i)\}\right|    \\
  &\quad\leq |\{1\leq i\leq m:sx\in D_i\}|+|\{1\leq i\leq m:x\in t_i^{-1}(sH\setminus C_i)\}|   \\
  &\quad\leq 2\delta\sqrt n
    =2\delta\sqrt{dm}=\varepsilon\sqrt m.
\end{align*}
Here the two terms are bounded using the original multiplicity estimate at $sx$ and at $x$, respectively. Since the construction works for every $m\geq\lceil N/d\rceil$, this proves that $H$ has property ${\rm {PHP}}$.
\end{proof}

\begin{rem}
    One can easily modify the proof above to conclude that finite-index inclusions $H\leq G$ have property ${\rm PHP}$ in the sense of \cite{FKOCP26}. Thus, appealing to \cite{FKOCP26}, we see that under the property ${\rm PHP}$, finite-index inclusions of groups give rise to selfless inclusions of \Cs-algebras. The required modifications involve assuming that $F\subseteq G$ and changing the sets $A_i, B_i$ to $s^{-1} C_i$ and $s^{-1}D_i$.
\end{rem}

\section{Proof of the main results}

\begin{lem}\label{exp}
Let $R$ be a domain, and let $G$ be a group with property ${\rm {PHP}}$. Then, for every non-zero $\alpha\in RG$, there is an element $\theta\in\alpha RG$ such that
$$
 |{\rm supp}(\theta\beta)|>|{\rm supp}(\beta)|
 \qquad\text{for every }0\neq\beta\in RG.
$$
\end{lem}

\begin{proof}
Write
$$
    \alpha=\sum_{a\in F}r_a a,\qquad 0\neq r_a\in R,
$$
where $F={\rm supp}(\alpha)$ is finite and non-empty. Apply the definition of ${\rm {PHP}}$ to $F$ with $\varepsilon=1/8$ and let $n,t_i,C_i,D_i$ be the resulting witnesses. We use them to set
$$
 z=\sum_{i=1}^n(t_i+t_i^{-1}),
 \qquad \theta=\alpha z\in\alpha RG.
$$
Now, fix $0\neq\beta\in RG$ and write $\beta=\sum_{h\in H}s_hh$, where $H={\rm supp}(\beta)$.  We regard the expansion
\begin{equation}\label{expansion}
    \theta\beta=\sum_{a\in F}\sum_{i=1}^n\sum_{h\in H} r_as_h\bigl(at_ih+at_i^{-1}h\bigr)
\end{equation}
as a sum indexed by quadruples $(a,i,\sigma,h)$, where $\sigma\in\{+,-\}$. For each $h\in H$, declare both signs attached to $i$ to be good when $i\notin J(h)$, and declare both signs bad otherwise. Then, there are at least $2|F|(n-\varepsilon\sqrt n)|H|$ good term indices and at most $2|F|\varepsilon\sqrt n|H|$ bad term indices.

Note that the group elements represented by distinct good term indices are
distinct. Indeed, if $i\notin J(h)$, then
$$
 t_ih\in C_i,\qquad t_i^{-1}h\in t_i^{-1}( G\setminus D_i).
$$
Thus, the plus and minus terms lie respectively in $aC_i$ and $at_i^{-1}( G\setminus D_i)$, which are disjoint. Hence, if a group element shares two good $(a,i,\sigma,h)$-labels, then these must be equal. 

Now, observe that a group element with a good $(a,i,\sigma,h)$-label could also be the result of a bad $(a,i,\sigma,h)$-tuple. If this happens, we will call it contaminated. Because the map $(a,i,\sigma,h)\mapsto at_i^\sigma h$ is injective on good labels, each bad label can only contaminate a single good value. Hence, the number of uncontaminated good values is at least
$$
 2|F|(n-2\varepsilon\sqrt n)|H|.
$$
Each such value occurs exactly once in \eqref{expansion}; its coefficient is therefore $r_as_h$, which is non-zero because $R$ is a domain. We conclude that every uncontaminated good value lies in ${\rm supp}(\theta\beta)$.

Finally, using that $n\geq1$ and $\varepsilon=1/8$, we see that
$$
 |{\rm supp}(\theta\beta)|\geq 2|F|(n-2\varepsilon\sqrt n)|H|=2|F|\left(n-\tfrac14\sqrt n\right)|H| \geq\tfrac32|F|n|H|>|H|,
$$
proving the result.
\end{proof}

The next lemma is a direct consequence of a result of Alexander-Nishinaka \cite{AlNi17}.

\begin{lem}\label{AN}
Let $R$ be a ring with identity, and let $G$ be a group that contains a free subgroup with free basis given by
$$
 \mathcal X=\{x_{j,s}:j\in J,\ s\in\{1,2,3\}\}.
$$
If $J_0\subseteq J$ is finite and non-empty and $0\neq v_j\in R G$ for $j\in J_0$, then
$$
 \Big|
 {\rm supp}\Big(\sum_{j\in J_0}
 (x_{j,1}+x_{j,2}+x_{j,3})v_j\Big)
 \Big|
 >\sum_{j\in J_0}|{\rm supp}(v_j)|.
$$
\end{lem}

\begin{proof}
Enumerate $J_0=\{j_1,\ldots,j_n\}$ and set
$$
 S_i={\rm supp}(v_{j_i}),\qquad
 X_i=\{x_{j_i,1},x_{j_i,2},x_{j_i,3}\},\qquad
 m=\sum_{i=1}^n|S_i|.
$$
The elements of $\bigcup_iX_i$ are mutually reduced in the sense of Alexander-Nishinaka because they belong to a free basis. Applying their isolated-product estimate \cite[Lemma 3.5]{AlNi17} to $V=\bigcup_i(X_i\times S_i)$ shows that more than $m$ pairs $(x,h)\in V$ have the property that $xh\neq x'h'$ for every other $(x',h')\in V$.

Each such pair contributes a distinct group element to the sum $\sum_{j\in J_0}
 (x_{j,1}+x_{j,2}+x_{j,3})v_j$, and its coefficient is exactly the non-zero coefficient of $h$ in $v_{j_i}$ when $x\in X_i$. Hence, none of these terms cancels, and the support of the sum has cardinality greater than $m$, as required.
\end{proof}

The following lemma is a classical criterion due to Formanek. It can be found in \cite[Theorem 1]{Fo73}. We include its proof for the reader's sake.

\begin{lem}[Formanek's criterion]\label{formanek}
Let $A$ be a unital ring. Suppose there is a proper right ideal $\rho\subsetneq A$ such that $\rho+I=A$ for every non-zero two-sided ideal $I\triangleleft A$. Then $A$ is right primitive.
\end{lem}

\begin{proof}
Choose a maximal right ideal $M$ containing $\rho$.  If a non-zero two-sided ideal $I$ were contained in $M$, then $A=\rho+I\subseteq M$, which is a contradiction. Hence, $M$ contains no non-zero
two-sided ideal.

The right module $A/M$ is simple. Its annihilator is
$$
 \{a\in A:Aa\subseteq M\},
$$
which is the largest two-sided ideal contained in $M$. This annihilator is zero, so $A/M$ is faithful. \end{proof}

\begin{thm}\label{domain}
Let $G$ be a group with property ${\rm {PHP}}$. Suppose that $G$ contains a non-abelian free subgroup of rank $|G|$.  Then, for every domain $R$ with $|R|\leq|G|$, the group ring $R G$ is primitive.
\end{thm}

\begin{proof}
For simplicity, set $A=R G$. Since $G$ is infinite, we have $|R|\leq|A|=|G|$. Choose, inside the given free subgroup, distinct free-basis elements $x_{\alpha,s}$, for $0\neq\alpha\in A$ and $s\in\{1,2,3\}$. Furthermore, for each $0\neq\alpha\in A$, appealing to Lemma \ref{exp}, we find an element $\theta_\alpha\in\alpha A$ such that $|{\rm supp}(\theta_\alpha\beta)|>|{\rm supp}(\beta)|$, for all $0\neq\beta\in A$. Define
$$
 e_\alpha=(x_{\alpha,1}+x_{\alpha,2}+x_{\alpha,3})\theta_\alpha.
$$
Because $\theta_\alpha\in\alpha A$, we have $e_\alpha\in A\alpha A$.  Let
$$
 \rho=\sum_{0\neq\alpha\in A}(1+e_\alpha)A
$$
be the right ideal generated by the $1+e_\alpha$'s. We claim that $\rho$ is proper.  Suppose, to the contrary, that $1\in\rho$. Then, for a finite non-empty set $J_0\subseteq A\setminus\{0\}$, there are non-zero $u_\alpha\in A$ such
that
$$
 1=\sum_{\alpha\in J_0}(1+e_\alpha)u_\alpha.
$$
Zero summands, if any, have been removed.  Put
$$
 \qquad W=\sum_{\alpha\in J_0}e_\alpha u_\alpha, \qquad U=\sum_{\alpha\in J_0}u_\alpha,
$$
and observe that $W+U=\sum_{\alpha\in J_0}(1+e_\alpha)u_\alpha=1$. Then, $\theta_\alpha u_\alpha$ is non-zero and, by Lemma \ref{AN}, we have that 
$$
 |{\rm supp}(W)|>\sum_{\alpha\in J_0}|{\rm supp}(\theta_\alpha u_\alpha)|.
$$
Since ${\rm supp}(W+U)$ contains
${\rm supp}(W)\setminus{\rm supp}(U)$ and
$|{\rm supp}(U)|\leq\sum_\alpha|{\rm supp}(u_\alpha)|$, we obtain
\begin{align*}
 |{\rm supp}(W+U)|
 &\geq |{\rm supp}(W)|-|{\rm supp}(U)|\\
 &>\sum_{\alpha\in J_0}
   \bigl(|{\rm supp}(\theta_\alpha u_\alpha)|-|{\rm supp}(u_\alpha)|\bigr)\\
 &\overset{\ref{exp}}{\geq} |J_0|\geq1.
\end{align*}
Thus $W+U$ has at least two points in its support. This shows that $\rho$ is proper.

Now let $I\triangleleft A$ be a non-zero two-sided ideal and choose $0\neq\alpha\in I$.  Since $e_\alpha\in A\alpha A\subseteq I$ and $1+e_\alpha\in\rho$, we have
$$
 1=(1+e_\alpha)-e_\alpha\in\rho+I.
$$
Formanek's criterion applies, and it implies that $A$ is (right) primitive.
\end{proof}

Recall that the FC-center of a group $G$ is given by
$$
 {\rm FC}(G)=\{g\in G:|\{hgh^{-1}:h\in G\}|<\infty\}.
$$
By \cite[Theorem 14]{Oz25}, a group $G$ that enjoys property ${\rm {PHP}}$ gives rise to a selfless $\rm C^*$-algebra (with respect to the canonical tracial state). Since selfless $\rm C^*$-algebras are simple \cite[Theorem 3.1]{Ro25}, we necessarily have that ${\rm FC}(G)=\{1\}$, as the FC-center is an amenable normal subgroup of $G$ \cite[Theorem 1.3]{BKKO17}.

We are now ready to prove the remaining part of our main theorem. In order to do so, we recall the following theorem of Passman: if $kG$ is primitive and $K/k$ is a field extension, then $KG$ is primitive provided either $K/k$ is algebraic or ${\rm FC}(G)=\{1\}$ \cite[Theorem 2]{Pa73}. 

\begin{thm}\label{mainnthm}
   Let $G$ be a group with Ozawa's property ${\rm PHP}$ that contains a non-abelian free subgroup whose cardinality is the same as that of $G$. Then, if $R$ is a domain with $|R|\leq |G|$, the group ring $RG$ of $G$ over $R$ is primitive. Moreover, the group ring $KG$ is primitive for any field $K$.
\end{thm}

\begin{proof}
The case of $R$ being a countable domain was settled in Theorem \ref{domain}. Let $K$ be any field and let $k$ be its prime subfield. Thus, $k$ is either finite or isomorphic to $\mathbb Q$. In any case, $k$ is a countable field, and we have that $kG$ is primitive. Because of the discussion above, we have that ${\rm FC}(G)=\{1\}$, and that $KG$ is primitive, due to Passman's result \cite[Theorem 2]{Pa73}. 
\end{proof}

We now prove the remaining theorems that we stated in the introduction.

\begin{thm}
Let $G$ be a non-trivial countable linear group with trivial amenable radical. Then, if $R$ is a countable domain, the group ring $RG$ is primitive. Moreover, the group ring $KG$ is primitive for any field $K$.
\end{thm}

\begin{proof}
Vigdorovich proves that $ G$ is a directed union of finitely generated linear subgroups with trivial amenable radical \cite[Lemma 5.3]{Vi26}. For each such subgroup, the proof of \cite[Theorem 1.2]{Vi26} constructs a faithful action on a flag space and deduces the property ${\rm {PHP}}$ as a consequence.  Hence, Lemma \ref{directed} gives ${\rm {PHP}}$ for $G$.

In order to appeal to Theorem \ref{mainthm}, it only remains to verify the free-subgroup hypothesis. Indeed, if $G$ contained no copy of $\mathbb F_2$, then by the Tits alternative, every finitely generated subgroup of $G$ would be virtually solvable; hence, amenable \cite{Ti72}.  A directed union of amenable discrete groups is amenable, so $ G$ itself would be amenable. Hence, $G$ contains $\mathbb F_2$, and Theorem \ref{mainnthm} applies.
\end{proof}

\begin{thm}
Let $G$ be a countable vigorous group. Then, if $R$ is a countable domain, the group ring $RG$ is primitive. Moreover, the group ring $KG$ is primitive for any field $K$.
\end{thm}
\begin{proof}
    Vigorous groups have property $\rm PHP$ by \cite[Theorem A.2]{Ar26}. They contain non-abelian free subgroups by \cite[Theorem 3.3]{BlElHy24}. The conclusion follows from Theorem \ref{mainnthm}.
\end{proof}

\begin{thm}
   Let $G$ be a countable group that admits a topologically free, extremely proximal action $G\curvearrowright X$. Then, if $R$ is a countable domain, the group ring $RG$ is primitive. Moreover, the group ring $KG$ is primitive for any field $K$.
\end{thm}
\begin{proof}
    These groups have property $\rm PHP$ by \cite[Proposition 15]{Oz25}. They contain non-abelian free subgroups by a theorem of Glasner \cite[Theorem 3.4]{Gl74}. Once again, the conclusion follows from Theorem \ref{mainnthm}.
\end{proof}

\section*{AI statement}

GPT-5.6 Sol was used to help with English language editing, proofreading, and grammatical corrections. All of this article has been carefully written by the author. Every mathematical idea behind this work was conceived and worked out by the author.

\section*{Acknowledgments}

The author gratefully acknowledges support from the Simons Foundation Dissertation Fellowship SFI-MPS-SDF-00015100. He also thanks Ben Hayes and Aaratrick Basu for all the interesting discussions surrounding the present topic.

\printbibliography

\bigskip
\bigskip
ADDRESS

\smallskip
\smallskip
Felipe I. Flores

Department of Mathematics, University of Virginia,

114 Kerchof Hall. 141 Cabell Dr,

Charlottesville, Virginia, United States

E-mail: hmy3tf@virginia.edu
\end{document}